\documentclass{amsart}
\usepackage{amsfonts}
\usepackage{amsmath}

\newtheorem{theorem}{Theorem}
\theoremstyle{plain}

\newtheorem{proposition}{Proposition}

\numberwithin{equation}{section}

\begin{document}
\title{Degenerations and Contractions of Algebras and Forms }
\author{Harold N. Ward}
\date{May 2023}
\address{Department of Mathematics\\
University of Virginia\\
Charlottesville, VA 22904\\
USA}
\email{hnw@virginia.edu}

\begin{abstract}
This note intertwines the concepts of degeneration and contraction of
algebras and quadratic forms defined on a vector space $V$. The general
linear group $\mathrm{GL}(V)$ acts regularly on the spaces of these two
objects. The base field is taken to be infinite of characteristic not 2. It
is unrestricted otherwise, as in the first cited paper of Ivanova and
Pallikaros. We apply the results to algebras and forms in dimensions two and
three.
\end{abstract}

\maketitle

\section{Introduction}

The framework for what follows involves an infinite field $\mathbb{F}$ and a
vector space $V$ over $\mathbb{F}$ with $\dim V=n$. Let $\mathrm{G}$ stand
for the general linear group $\mathrm{GL}(V)$. The members of $\mathrm{G}$
will usually be given by $n\times n$ nonsingular matrices having
coefficients in $\mathbb{F}$ relative to some basis in $V$. These will act
by left multiplication on column vectors representing the members of $V$,
and $\mathbf{g}\in \mathrm{G}$ is usually taken to mean its matrix. Let $%
\Lambda $ be a second finite-dimensional vector space over $\mathbb{F}$ upon
which $\mathrm{G}$ acts by linear transformations whose matrices relative to
a given basis of $\Lambda $ have entries in the polynomial ring $\mathbb{F}%
[g_{11},g_{12},\ldots ,g_{nn},1/d]$. Here the $g_{ij}$ are the matrix
entries of $\mathbf{g}$ and $d=\det \mathbf{g}$. We refer to this action as
the \textbf{companion} action on $\Lambda $, and write it as $\Lambda (%
\mathbf{g})$. The image of $\mathbf{\lambda }$ under $\mathbf{g}$ will be
written as $\mathbf{\lambda g}$, this time with action on the right. The
action matrix coefficients are \textbf{regular functions} \cite[Section 1.1.2%
]{GW}, and the action itself will be called regular. We shall be interested
in the $\mathrm{G}$-orbits on $\Lambda $ and their Zariski closures.

The main example is that in which $\Lambda $ is the space of structure
vectors for nonassociative $\mathbb{F}$-algebras defined on $V$ \cite{IP1}.
For such an algebra, the \textbf{structure vector} of $\mathfrak{a}$
relative to a given basis $\mathbf{v}_{1},\ldots ,\mathbf{v}_{n}$ of $V$ is
the vector $\mathbf{\lambda }=\mathbf{\lambda }(\mathfrak{a})$ whose
components $\lambda _{ijk}$ present the product $\left[ \,,\,\right] $ in $%
\mathfrak{a}$:%
\begin{equation*}
\lbrack \mathbf{v}_{i},\mathbf{v}_{j}]=\sum_{k=1}^{n}\lambda _{ijk}\mathbf{v}%
_{k}.
\end{equation*}%
$\mathrm{G}$ acts on the set of algebras: the product for the image algebra $%
\mathfrak{a}^{\prime }=\mathfrak{a}\mathbf{g}$ under $\mathbf{g}\in \mathrm{G%
}$ is given by $[\mathbf{u},\mathbf{v}]^{\prime }=\mathbf{g}^{-1}[\mathbf{gu}%
,\mathbf{gv}]$. This presents $\mathbf{g}$ as an isomorphism from $\mathfrak{%
a}^{\prime }$ to $\mathfrak{a}$. The corresponding action in $\Lambda $ is
defined by $\mathbf{\lambda }(\mathfrak{a})\mathbf{g}=\mathbf{\lambda }(%
\mathfrak{a}^{\prime })$. Writing this out yields the \textbf{%
change-of-basis }formula: if $\mathbf{\lambda }^{\prime }=\mathbf{\lambda }(%
\mathfrak{a}^{\prime })$, then%
\begin{equation*}
\lambda _{ijk}^{\prime }=\sum_{a,b,c}g_{ai}g_{bj}\widehat{g}_{kc}\lambda
_{abc},
\end{equation*}%
the $\widehat{g}_{kc}$ being matrix entries of $\mathbf{g}^{-1}$. As those
are $1/\det \mathbf{g}$ times entries in the classical adjoint of $\mathbf{g}
$, the coefficients of the change of basis action are indeed members of $%
\mathbb{F}[g_{11},g_{12},\ldots ,g_{nn},1/d]$. The name \textquotedblleft
change-of-basis\textquotedblright\ arises from the fact that $\mathbf{%
\lambda }^{\prime }$ is also the structure vector for $\mathfrak{a}$
relative to the basis $\mathbf{gv}_{1},\ldots ,\mathbf{gv}_{n}$.

In this context, the members of the Zariski closure $\overline{\mathbf{%
\lambda }\mathrm{G}}$ of the $\mathrm{G}$-orbit $\mathbf{\lambda }\mathrm{G}$
are called \textbf{degenerations} of $\mathbf{\lambda }$. There is a large
literature on algebra degenerations, most of which focusses on particular
types (Lie and Jordan algebras, for example). Some of this is cited in \cite%
{IP1}. In most other papers, $\mathbb{F}$ is taken to be the real or complex
field.

\section{Contractions}

For a general $\Lambda $, we shall call any member of $\overline{\mathbf{%
\lambda }\mathrm{G}}$ a degeneration of $\mathbf{\lambda }$. If $\mathbf{%
\lambda }^{\prime }\in \overline{\mathbf{\lambda }\mathrm{G}}$, the standard
notation is $\mathbf{\lambda }\longrightarrow \mathbf{\lambda }^{\prime }$,
and $\mathfrak{a}\longrightarrow \mathfrak{a}^{\prime }$ if $\mathbf{\lambda 
}(\mathfrak{a})=\mathbf{\lambda }$ and $\mathbf{\lambda }^{\prime }(%
\mathfrak{a})=\mathbf{\lambda }^{\prime }$ for algebras $\mathfrak{a}$ and $%
\mathfrak{a}^{\prime }$. One way to obtain degenerations is by the process
of \textbf{contraction }(a term sometimes used interchangeably with
\textquotedblleft degeneration\textquotedblright ). Let $\mathcal{C}=\{%
\mathbf{g}^{t}|t\in \mathbb{F}\}$ be a one-parameter family of members of $%
\mathrm{G}$--a \textbf{contraction family}. The matrix entries of $\mathbf{g}%
^{t}$ are to be in the field $\mathbb{F}(t)$ of rational functions in $t$
over $\mathbb{F}$. We may write a function $f(t)$ with $t$ as an exponent: $%
f^{t}$. Define the \textbf{order} $\mathrm{ord}f$ of $f(t)$ to be $\infty $
if $f=0$, and for $f(t)=t^{m}a(t)/b(t)$ with $a(0)$ and $b(0)$ both nonzero,
put $\mathrm{ord}f=m$. If $\mathrm{ord}f=0$, $f$ will be called a unit; it
is indeed a unit of the subalgebra of $\mathbb{F}(t)$ consisting of the
functions with nonnegative order (an algebra containing $\mathbb{F}[t]$).
The order presents an exponential discrete valuation of $\mathbb{F}(t)$
whose completion is the field $\mathbb{F}((t))$ of formal Laurent series
over $\mathbb{F}$. Some of the theorems in \cite{L} that will be used are
stated in terms of this completion. If $\mathrm{ord}f\geq 0$, then $t=0$ can
safely be substituted in $f$. We usually write $f^{0}$ instead of $f(0)$,
and say that $f^{0}$ \textbf{exists}, picturing $f^{0}$ as a kind of limit.

For a matrix $M$ of any size with entries in $\mathbb{F}(t)$, let $\mathrm{%
ord}M$ be the minimum of the orders of the entries in $M$. Then $M$ can be
evaluated at $t=0$ just when $\mathrm{ord}M\geq 0$. The rationality demand
on the entries in $\mathbf{g}^{t}$ is independent of the basis of $V$
chosen, and the entries in the companion action on $\Lambda $ will also be
in $\mathbb{F}(t)$. When $\mathcal{C}$ is understood and $\mathbf{\lambda }%
\in \Lambda $, let $\mathbf{\lambda }^{t}=\mathbf{\lambda g}^{t}$. If $%
\mathbf{\lambda }$ is a structure vector for an algebra $\mathfrak{a}$, and $%
\mathbf{\lambda }^{0}$ exists, we shall also write $\mathfrak{a}^{0}$ for
the algebra corresponding to $\mathbf{\lambda }^{0}$.

\begin{proposition}
\label{PropConInCl}Suppose that for some $\mathbf{\lambda }\in \Lambda $, $%
\mathrm{ord}(\mathbf{\lambda }^{t})\geq 0$, so that $\mathbf{\lambda }^{0}$
exists. Then $\mathbf{\lambda }^{0}\in \overline{\mathbf{\lambda }\mathrm{G}}%
.$
\end{proposition}

\begin{proof}
Let $P$ be a polynomial function in the coordinates on $\Lambda $ relative
to a basis for which the entries of $\mathbf{\lambda }^{t}\ $have
nonnegative orders. Then $\mathrm{ord}(P(\mathbf{\lambda }^{t}))\geq 0$.
Thus $P(\mathbf{\lambda }^{t})$ is either identically $0$ or it is $%
t^{m}u^{t}$, $m\geq 0$, for some unit $u^{t}$. If $P$ is $0$ on the orbit $%
\mathbf{\lambda }$\textrm{$G$}, then $P(\mathbf{\lambda }^{t})$ is $0$ for
all values of $t$ for which $\det \Lambda (\mathbf{g}^{t})$ is not $0$, an
infinite set. So either $P(\mathbf{\lambda }^{t})$ is identically $0$ or $%
m>0 $. In either event, $P(\mathbf{\lambda }^{0})=0$. This being true for
any such $P$, we conclude that $\mathbf{\lambda }^{0}\in \overline{\mathbf{%
\lambda }\mathrm{G}}$.
\end{proof}

The argument here is used for most of the degeneration results in the
literature, and some examples occur in \cite{IP1}. We refer to the limit\ $%
\mathbf{\lambda }^{0}$ as a \textbf{contraction} of $\mathbf{\lambda }$ and
say that $\mathbf{\lambda }$ is amenable to contraction (by the family $%
\mathcal{C}$). The term \textquotedblleft contraction\textquotedblright\ has
been widely used to suggest a limiting process for algebras. The two surveys 
\cite{B} and \cite{N} give some insight into motivation, and \cite{IP1}
contains historical background along with other exposition.

A favorite example has $\mathbf{g}^{t}=t\mathrm{I}$, $\mathrm{I}$ the
identity matrix; $\mathbf{g}^{-1}=t^{-1}\mathrm{I}$. Then in the
change-of-basis action,%
\begin{eqnarray*}
\lambda _{ijk}^{t} &=&\sum_{a,b,c}(t\delta _{ai})(t\delta
_{bj})(t^{-1}\delta _{kc})\lambda _{abc} \\
&=&t\lambda _{ijk}.
\end{eqnarray*}%
Thus $\lambda _{ijk}^{0}=0$, and $\mathbf{\lambda }^{0}$ is the structure
vector of the algebra for which all the products are $0$. So this
\textquotedblleft Abelian algebra\textquotedblright\ is a degeneration of
every algebra.

\section{Quadratic Forms}

In this section we consider the space $\mathfrak{Q}$ of quadratic forms on $%
V $ and the action of \textrm{$G$} on it. By and large we use the notation
from \cite[Chapter 1]{L}. As there, we take $\mathrm{char}\mathbb{F}\neq 2$.
Let $Q$ be a quadratic form on $V$ with $B_{Q}$ the corresponding symmetric
bilinear form: $Q(\mathbf{u})=B_{Q}(\mathbf{u},\mathbf{u})$ and $B_{Q}(%
\mathbf{u},\mathbf{v})=\frac{1}{4}(Q(\mathbf{u}+\mathbf{v})-Q(\mathbf{u}-%
\mathbf{v}))$. (We'll generally drop the $Q$ on $B_{Q}$.) The matrix $\left[
B\right] $ of $B$ relative to a basis $\mathbf{v}_{1},\ldots ,\mathbf{v}_{n}$
of $V$ has entries $B(\mathbf{v}_{i},\mathbf{v}_{j})$. The right action of 
\textrm{$G$} on $\mathfrak{Q}$ is given by $(Q\mathbf{g})(\mathbf{u})=Q(%
\mathbf{gu})$ and $B\mathbf{g}(\mathbf{u},\mathbf{v})=B(\mathbf{gu},\mathbf{%
gv})$. For the matrices, $\left[ B\mathbf{g}\right] =\mathbf{g}^{\mathrm{T}}%
\left[ B\right] \mathbf{g}$, $\mathbf{g}^{\mathrm{T}}$ the transpose of $%
\mathbf{g}$. The forms $Q$ and $Q\mathbf{g}$ are called \textbf{equivalent},
written $Q\cong Q\mathbf{g}$. To show that $V$ is endowed with a quadratic
form $Q$, one writes $(V,Q)$ or $(V,B_{Q})$ and speaks of $(V,Q)$ as a 
\textbf{quadratic space}. The \textbf{rank} $\mathrm{rank}Q$ of $Q$ is the
rank of $\left[ B_{Q}\right] $, and $Q$ is called \textbf{nonsingular}
(\textquotedblleft regular\textquotedblright\ in \cite{L}) if this rank is $%
\dim V$. A $0$-form is one of rank $0$, possibly with $\dim V=0$.

If $(V_{1},Q_{1}),\ldots ,(V_{m},Q_{m})$ are quadratic spaces, their \textbf{%
orthogonal sum} is the direct sum of the $V_{i}$ endowed with the quadratic
form $Q$ given by $Q(\mathbf{u}_{1},\ldots ,\mathbf{u}_{m})=Q_{1}(\mathbf{u}%
_{1})+\ldots +Q_{m}(\mathbf{u}_{m})$, where $\mathbf{u}_{i}\in V_{i}$. One
writes $V=V_{1}\perp \ldots \perp V_{m}$, and $Q=Q_{1}\perp \ldots Q_{m}$.
When all the $V_{i}$ have dimension $1$ and $V_{i}=\mathbb{F}\mathbf{v}_{i}$%
, with $Q_{i}(\mathbf{v}_{i})=x_{i}$, $V$ is displayed as $V=\left\langle
x_{1},\ldots ,x_{m}\right\rangle $. Relative to the basis $\mathbf{v}%
_{1},\ldots ,\mathbf{v}_{m}$, $\left[ B_{Q}\right] $ is diagonal, and $Q$
itself is also said to be diagonal. The \textbf{scaled} form $\left\langle
xx_{1},\ldots ,xx_{m}\right\rangle $ is denoted $\left\langle x\right\rangle
\left\langle x_{1},\ldots ,x_{m}\right\rangle $ in \cite{L}, employing a
Kronecker product. However, we can safely use $x\left\langle x_{1},\ldots
,x_{m}\right\rangle $. A $0$-form will be $\left\langle 0,0,\ldots
,0\right\rangle $ when $\dim V>0$.

Every quadratic form is equivalent to a diagonal form \cite[I.2.4]{L}. A
major result is \textbf{Witt's Cancellation Theorem} \cite[I.4.2]{L}: if $%
Q\perp Q_{1}\cong Q\perp Q_{2}$, then $Q_{1}\cong Q_{2}$. Each quadratic
form $Q$ is equivalent to a form $Q_{N}\perp Q_{Z}$ in which $Q_{N}$ is
nonsingular and $Q_{Z}$ is a $0$-form (corresponding to the \textbf{radical}
of $Q$ \cite[p. 5]{L}). In two such decompositions, the nonsingular parts
are equivalent by the cancellation theorem. One says that $Q$ \textbf{%
represents }$Q^{\prime }$ if $Q$ is equivalent to $Q_{N}^{\prime }\perp
Q^{\prime \prime }$ for some form $Q^{\prime \prime }$ \cite[Section 42]{OM}%
. Proper basis choice gives $\left[ B_{Q}\right] =%
\begin{bmatrix}
B_{Q_{N}^{\prime }} & 0 \\ 
0 & B_{Q^{\prime \prime }}%
\end{bmatrix}%
$. Representation depends only on the equivalence classes of $Q$ and $%
Q^{\prime }$.

There is a further refinement of nonsingular quadratic forms: for a form $Q$%
, an \textbf{isotropic} member $\mathbf{v}$ of $V$ has $\mathbf{v}\neq 0$
but $Q(\mathbf{v})=0$ \cite[Sectin 1.4]{L}. The \textbf{hyperbolic plane} is 
$\left\langle 1,-1\right\rangle $, and an orthogonal sum of hyperbolic
planes is a \textbf{hyperbolic space}. \textbf{Anisotropc} spaces contain no
isotropic vector. The nonsingular form $Q_{N}$ above is equivalent to $%
Q_{A}\perp Q_{H}$, with $Q_{A}$ anisotropic and $Q_{H}$ hyperbolic; these
components of $Q$ are also unique to equivalence.

\subsection{Degenerations of quadratic forms}

What about degenerations of quadratic forms relative to the action of 
\textrm{$G$} on $\mathfrak{Q}$? With a basis chosen for $V$, coordinates for 
$\mathfrak{Q}$ can be taken to be those matrix entries $b_{ij}$ of the
matrices $[B_{Q}]$ for which $i\leq j$, because $\left[ B_{Q}\right] $ is
symmetric. If $P$ is a polynomial in the $b_{ij}$, let $P(Q)$ mean $P$
evaluated at the matrix entries of $[B_{Q}].$ The induced \textrm{$G$}%
-action is regular all right. Let $\mathfrak{Q}_{r}$ be the set of quadratic
forms having rank at most $r$; $\mathfrak{Q}_{r}$ is \textrm{$G$}-invariant.
As in \cite[Remark 3.15]{IP1}, $\mathfrak{Q}_{r}$ is Zariski closed. Here is
the degeneration result:

\begin{theorem}
\label{ThmOrbClQF}If $Q$ is a quadratic form of rank $r$, the orbit closure $%
\overline{Q\mathrm{G}}$ is $\mathfrak{Q}_{r}$.
\end{theorem}

\begin{proof}
This being true when $r=0$, let $r>0$. Relative to the chosen basis of $V$,
let $\left\langle x_{1},\ldots ,x_{r},0,\ldots ,0\right\rangle $ be a
diagonal form with nonzero $x_{i}$ that is equivalent to $Q$ and so in $Q%
\mathrm{G}$. If $P$ is a polynomial vanishing on $Q\mathrm{G}$, let $p(x)$
be the function $x\longrightarrow P(\left\langle x,x_{2},\ldots
,x_{r},0,\ldots ,0\right\rangle )$. Then $p(x)$ is a polynomial in $x$. For $%
z\neq 0$, the form $\left\langle z^{2}x_{1},x_{2},\ldots ,x_{m},0,\ldots
,0\right\rangle $ is equivalent to $Q$, by the diagonal map $%
v_{1}\longrightarrow zv_{1},\,v_{i}\longrightarrow v_{i}$ for $i\geq 2$.
That form then being in $Q\mathrm{G}$, $p(z^{2}x_{1})=0$. So $p(x)$ has
infinitely many $0$'s, making $p(x)$ identically $0$. That is, $%
P(\left\langle x,x_{2},\ldots ,x_{r},0,\ldots ,0\right\rangle )=0$ for all $%
x $. As this holds for any $P$ vanishing on $Q\mathrm{G}$, it must be that
all the forms $\left\langle x,x_{2},\ldots ,x_{r},0,\ldots ,0\right\rangle $
are in $\overline{Q\mathrm{G}}$.

For arbitrary $x_{1}^{\prime }$, we can apply this argument to $\left\langle
x_{1}^{\prime },x,x_{3},\ldots ,x_{r},0,\ldots ,0\right\rangle $, which is
in $\overline{Q\mathrm{G}}$ for $x=x_{2}$, and conclude that $\left\langle
x_{1}^{\prime },x_{2}^{\prime },x_{3}\ldots ,x_{r},0,\ldots ,0\right\rangle
\in \overline{Q\mathrm{G}}$ for all $x_{1}^{\prime },x_{2}^{\prime }$. And
so on: the result is that $\left\langle x_{1}^{\prime },x_{2}^{\prime
},\ldots ,x_{r}^{\prime },0,\ldots ,0\right\rangle \in \overline{Q\mathrm{G}}
$ for all $x_{1}^{\prime },\ldots ,x_{r}^{\prime }$ Thus $\overline{Q\mathrm{%
G}}\supseteq \mathfrak{Q}_{r}$. As $Q\in \mathfrak{Q}_{r}$ and $\mathfrak{Q}%
_{r}$ is closed, it must be that $\overline{Q\mathrm{G}}=\mathfrak{Q}_{r}$.
\end{proof}

This result shows that quadratic forms cannot really be separated by
degeneration: $Q$ degenerates to $Q^{\prime }$ just when $\mathrm{\mathrm{%
rank}}Q^{\prime }\leq \mathrm{rank}Q$.

\subsection{Contractions of quadratic forms}

Let $\mathcal{C}=\left\{ \mathbf{g}^{t}|t\in \mathbb{F}\right\} $ be a
contraction family and consider its action on $\mathfrak{Q}$. By the formula 
$\left[ B_{Q\mathbf{g}}\right] =\mathbf{g}^{\mathrm{T}}\left[ B_{Q}\right] 
\mathbf{g}$, the action is regular. Put $Q^{t}=Q\mathbf{g}^{t}$ and $%
B^{t}=B_{Q^{t}}$. The form $Q$ is amenable to contraction just when $\mathrm{%
ord}\left[ B^{t}\right] \geq 0$, giving limits $B^{0}$ and $Q^{0}$.

\begin{theorem}
\label{ThmQFConRep}Let $Q$ and $Q^{\prime }$ be quadratic forms on $V$. Then 
$Q^{\prime }$ is a contraction of $Q$ if and only if $Q$ represents $%
Q^{\prime }$.
\end{theorem}

\begin{proof}
Suppose that $Q$ represents $Q^{\prime }$. As remarked above, proper choice
of basis gives $\left[ B_{Q}\right] =%
\begin{bmatrix}
B_{Q_{N}^{\prime }} & 0 \\ 
0 & B_{Q^{\prime \prime }}%
\end{bmatrix}%
$. Using the same basis and block sizes, let $\mathbf{g}^{t}=%
\begin{bmatrix}
\mathrm{I} & 0 \\ 
0 & t\mathrm{I}%
\end{bmatrix}%
$. Then $(\mathbf{g}^{t})^{\mathrm{T}}\left[ B_{Q}\right] \mathbf{g}^{t}=%
\begin{bmatrix}
B_{Q_{N}^{\prime }} & 0 \\ 
0 & t^{2}B_{Q^{\prime \prime }}%
\end{bmatrix}%
$, with limit $%
\begin{bmatrix}
B_{Q_{N}^{\prime }} & 0 \\ 
0 & 0%
\end{bmatrix}%
$. This is a matrix for $B_{Q^{\prime }}$, showing that $Q^{\prime }$ is a
contraction of $Q$.

For the converse, let $Q^{\prime }$ be a contraction of $Q$ by the family $%
\mathcal{C}$: $Q^{\prime }\cong Q^{0}$. With $B=B_{Q}$, we have $\left[ B^{t}%
\right] =(\mathbf{g}^{t})^{\mathrm{T}}\left[ B\right] \mathbf{g}^{t}$, and $%
\mathrm{ord}\left[ B^{t}\right] \geq 0$. We can assume that $[B]$ is
diagonal: if $\left[ \mathbf{s}\right] ^{\mathrm{T}}[B]\left[ \mathbf{s}%
\right] $ is diagonal for some $\mathbf{s}\in \mathrm{G}$, replace $\mathbf{g%
}^{t}$ by $\mathbf{s}^{-1}\mathbf{g}^{t}$. We can also assume that $\left[
B^{t}\right] $ is diagonal. To see that, follow the diagonalization process
outlined in \cite[Section 42G]{O}. Each step replaces a matrix $M$ with $E^{%
\mathrm{T}}ME$, this time with the elementary matrix $E$ having entries in $%
\mathbb{F}(t)$. But avoid taking $E$ to be a scaling matrix. Let $m$ be an
entry in $M$ of least (nonnegative) order. If $m$ is on the diagonal, we can
permute to bring it to the $(1,1)$ position. If $m$ is off-diagonal and the
diagonal entries all have orders strictly larger than $\mathrm{\mathrm{ord}}m
$, we can use a computation illustrated by the following $2\times 2$ example
to change the $(1,1)$ entry to an element of order $\mathrm{ord}m$:%
\begin{equation*}
\begin{bmatrix}
1 & 1 \\ 
0 & 1%
\end{bmatrix}%
\begin{bmatrix}
a & m \\ 
m & b%
\end{bmatrix}%
\begin{bmatrix}
1 & 0 \\ 
1 & 1%
\end{bmatrix}%
=%
\begin{bmatrix}
2m+a+b & b+m \\ 
b+m & b%
\end{bmatrix}%
.
\end{equation*}%
Since $a$ and $b$ have orders strictly larger than $\mathrm{ord}m$, $\mathrm{%
\mathrm{ord}}(2m+a+b)=\mathrm{ord}m$. Call the new $(1,1)$ entry $m$ again,
now with least order. Then the entries in the new matrix in row 1 and column
1, other than that at $(1,1)$, can be sent to $0$, using elementary matrices
with off-diagonal entry $-c/m$, where $c$ is the entry to be made $0$. Here $%
\mathrm{ord}(-c/m)\geq 0$. The process can be repeated on the submatrix in
the rows and columns with indices larger than 1. If the product of the
elementary matrices used is $\left[ \mathbf{e}\right] $, $\mathbf{e}$ can
then be incorporated into $\mathbf{g}^{t}$. Since $B$ and $B^{t}$ have the
same rank, we can also permute and remove all-0 rows and columns to assume
that $B$ and $B^{t}$ are nonsingular. Call the new rank $n$ again. The final
form for $\left[ B^{t}\right] $ is a block diagonal matrix whose blocks are
matrices of diagonal forms $D_{0}^{t},\,tD_{1}^{t},\,\ldots ,t^{s}D_{s}^{t}$%
, where each nonzero $\left[ D_{i}^{t}\right] $ is a diagonal matrix with
unit entries. Then $\left[ B^{0}\right] $ shows just the one block $\left[
D_{0}^{0}\right] $, the rest being $0$-blocks. Let $S_{i}^{t}$ be the
quadratic form corresponding to $D_{i}^{t}$. Finally, let $\mathcal{E}$ be
the set of even $i$ for which $D_{i}^{t}$ is nonzero, and $\mathcal{O}$ the
set of odd such $i$.

Now each nonsingular form $R$ over $\mathbb{F}(t)$ is $\mathbb{F}(t)$%
-equivalent to $R_{1}\perp tR_{2}$ where $R_{1}\cong \left\langle
u_{1},\ldots ,u_{r}\right\rangle $ and $R_{2}\cong \left\langle w_{1},\ldots
,w_{s}\right\rangle $, with the $u_{i}$ and $w_{j}$ units (one of the two
forms might not appear). This comes from scaling basis vectors by powers of $%
t$. The two forms $R_{1}$ and $R_{2}$ are not necessarily unique, but if $%
R\cong R_{1}^{\prime }\perp tR_{2}^{\prime }$ is another such decomposition,
then with superscript $0$'s indicating evaluation of the $u_{i}$ and $w_{i}$
at $t=0$, either $(R_{1}^{\prime })^{0}\cong R_{1}^{0}\perp H$ or else $%
(R_{1})^{0}\cong (R_{1}^{\prime })^{0}\perp H$, the $H$'s being hyperbolic
forms over $\mathbb{F}$. The same holds for $R_{2}$ and $R_{2}^{\prime }$
(see \cite[V.1]{L}, which presents a theorem of T. A. Springer involving the
Witt ring and the completion $\mathbb{F}((t))$).

For $Q$ in Theorem \ref{ThmQFConRep}, just $Q_{1}$ appears, with $Q_{1}\cong
Q$, since $Q$ is nonsingular with coefficients in $\mathbb{F}$. For $Q^{t}$,
we can take $Q_{1}^{t}$ to be the orthogonal sum $S_{\mathcal{E}}^{t}$ of
the $S_{i}^{t}$ with $i\in \mathcal{E}$ and $Q_{2}^{t}$ the sum with $i\in 
\mathcal{O}$, because of scaling basis vectors by powers of $t$. In the
hyperbolic statement, it must be that $Q\cong Q_{1}^{0}\perp H$ for some
hyperbolic form $H$ over $\mathbb{F}$, because $n\geq \mathrm{rank}Q_{1}^{0}$
and $n=\mathrm{rank}Q$. We have $Q_{1}^{0}=S_{\mathcal{E}}^{0}$, giving $%
Q\cong S_{\mathcal{E}}^{0}\perp H$. Now $Q^{0}=S_{0}^{0}\perp Z_{0}$ for
some $0$-form $Z_{0}$ of rank $n-\mathrm{rank}S_{0}^{0}$. If $Q^{0}$ itself
is a $0$-form, it is certainly represented by $Q$. If not, then $S_{0}^{0}$
is not a $0$-form, $Q_{N}^{0}\cong S_{0}^{0}$, and $Q\cong Q_{N}^{0}\perp R$
for some $R$, because $S_{0}^{0}$ is one of the orthogonal summands of $S_{%
\mathcal{E}}^{0}$. This all says that the limit $Q^{0}$ is represented by $Q$%
.
\end{proof}

\section{Trace forms}

Let $\mathfrak{a}$ be an algebra defined on the vector space $V$, with
structure vector $\mathbf{\lambda }$ relative to a basis $\mathbf{v}%
_{1},\ldots ,\mathbf{v}_{n}$ of $V$. For $\mathbf{u}\in V$, let $\mathrm{ad}%
_{\mathfrak{a}}(\mathbf{u})$ be the \textbf{adjoint} map $\mathbf{v}%
\longrightarrow \lbrack \mathbf{u},\mathbf{v}]$. Then the \textbf{trace form 
}of $\mathfrak{a}$ is the symmetric bilinear form defined by $T_{\mathfrak{a}%
}(\mathbf{u},\mathbf{v})=\mathrm{tr}(\mathrm{ad}_{\mathfrak{a}}(\mathbf{u})%
\mathrm{ad}_{\mathfrak{a}}(\mathbf{v}))$, with corresponding quadratic form $%
Q_{\mathfrak{a}}$. For $\left[ T_{\mathfrak{a}}\right] =\left[ t_{ij}\right] 
$, $t_{ij}=\sum_{k,l}\lambda _{ikl}\lambda _{jlk}$. The rank of $\left[ T_{%
\mathfrak{a}}\right] $ is termed the \textbf{rank} of $\mathfrak{a}$. This
rank will be at most $r$ exactly when all the determinants of $m\times m$
submatrices of $\left[ T_{\mathfrak{a}}\right] $ for $m>r$ are $0$. That
condition can be expressed in terms of polynomials in the $\lambda _{ijk}$,
so that the set of algebras of rank at most $r$ is Zariski closed (see, for
instance, \cite[Section 3]{IP1}). The book \cite{CP} presents a detailed
study of trace forms.

If $\mathcal{C=}\{\mathbf{g}^{t}|t\in \mathbb{F}\}$ is a contraction family
for $\mathbf{\lambda }$ and $\mathrm{ord}\mathbf{\lambda }^{t}\geq 0$, so
that $\mathbf{\lambda }^{0}$ exists, then $\mathrm{ord}T_{\mathfrak{a}}\geq
0 $ also, and $T_{\mathfrak{a}}^{0}$ exists too. Since $t_{ij}^{0}=%
\sum_{k,l}\lambda _{ikl}^{0}\lambda _{jlk}^{0}$, $T_{\mathfrak{a}}^{0}=T_{%
\mathfrak{a}^{0}}$ and $Q_{\mathfrak{a}}^{0}=Q_{\mathfrak{a}^{0}}$. By
Theorem \ref{ThmQFConRep}, $Q_{\mathfrak{a}}$ must represent $Q_{\mathfrak{a}%
^{0}}$.

\subsection{Two-dimensional algebras}

As an example, consider 2-dimensional commutative associative
algebras--fields in particular. Since $\mathrm{char}\mathbb{F}\neq 2$, every
quadratic field over $\mathbb{F}$ has the form $\mathbb{F}(w)$ with $w^{2}=s$
for some nonsquare $s\in \mathbb{F}$. The structure vector $\mathbf{\lambda }%
_{s}$ tailored to this by having $\mathbf{v}_{1}$ as the identity and $%
\mathbf{v}_{2}$ as $w$ is%
\begin{equation}
\lambda _{111}=1,\,\lambda _{122}=1,\,\lambda _{212}=1,\,\lambda _{221}=s,
\label{EqLamQuadField}
\end{equation}%
and all other $\lambda _{ijk}=0$. Then in the trace form $T_{s}$,%
\begin{equation*}
t_{ij}=\lambda _{i11}\lambda _{j11}+\lambda _{i12}\lambda _{j21}+\lambda
_{i21}\lambda _{j12}+\lambda _{i22}\lambda _{j22},
\end{equation*}%
and $\left[ T_{s}\right] =%
\begin{bmatrix}
2 & 0 \\ 
0 & 2s%
\end{bmatrix}%
$. For another nonsquare $s^{\prime }$, $Q_{s^{\prime }}\cong Q_{s}$ just
when $\left\langle 2s\right\rangle \cong \left\langle 2s^{\prime
}\right\rangle $, by Witt cancellation, and that in turn amounts to $%
s=s^{\prime }r^{2}$ for some $r\in \mathbb{F}$. And \emph{that} is the same
as $\mathbb{F}(w)\cong \mathbb{F}(w^{\prime })$, where $w^{\prime
2}=s^{^{\prime }}$. Now the only way for $Q_{s}$ to represent $Q_{s^{\prime
}}$ is by $Q_{s}\cong Q_{s^{\prime }}$. The upshot is that one quadratic
field can be a contraction of another only when the two fields are
isomorphic.

How about degeneracy? For $s\neq 0$, let $\mathfrak{f}_{s}$ be the
two-dimensional algebra with $\mathbf{\lambda }(\mathfrak{f}_{s})=\mathbf{%
\lambda }_{s}$ (\ref{EqLamQuadField}). As for fields, $\mathfrak{f}_{s}\cong 
\mathfrak{f}_{r^{2}s}$ when $r\neq 0$. Imitating the proof of Theorem \ref%
{ThmOrbClQF}, consider $\overline{\mathbf{\lambda }_{s}\mathrm{G}}$ and let $%
P$ be a polynomial vanishing on $\mathbf{\lambda }_{s}\mathrm{G}$. Define $%
p(x)$ by $x\longrightarrow P(\mathbf{\lambda }_{x})$. Then $p(r^{2}s)=0$ for
all $r\neq 0$, giving $p(x)$ an infinite number of $0$'s and making $p(x)$
identically $0$. As this applies to all such $P$, $\mathbf{\lambda }_{x}\in $
$\overline{\mathbf{\lambda }_{s}\mathrm{G}}$ for all $x$ (including $0$).
The fields and the direct sum $\mathbb{F}\oplus \mathbb{F}$ are the
semisimple 2-dimensional commutative algebras over $\mathbb{F}$. Moreover, $%
\mathbb{F}\oplus \mathbb{F}\cong \mathfrak{f}_{1}$: $\frac{1}{2}(\mathbf{v}%
_{2}+\mathbf{v}_{1})$ and $\frac{1}{2}(\mathbf{v}_{2}-\mathbf{v}_{1})$ are
orthogonal idempotents. So all these algebras degenerate to one another.

The approach in \cite[Section 3]{K} shows three other isomorphism types of
2-dimensional associative commutative algebras, with the structure vectors
displayed below (the algebra indexing follows \cite{IP2}). The last three
columns give the dimensions of the annihilator, the square, and the
derivation space of the algebras.%
\begin{equation*}
\begin{tabular}{cccccccccccc}
& $\lambda _{111}$ & $\lambda _{112}$ & $\lambda _{121}$ & $\lambda _{122}$
& $\lambda _{211}$ & $\lambda _{212}$ & $\lambda _{221}$ & $\lambda _{222}$
& ann & sq & der \\ 
$\mathfrak{a}_{0}$ & $0$ & $0$ & $0$ & $0$ & $0$ & $0$ & $0$ & $0$ & $2$ & $%
0 $ & $4$ \\ 
$\mathfrak{a}_{4}$ & $1$ & $0$ & $0$ & $1$ & $0$ & $1$ & $0$ & $0$ & $0$ & $%
2 $ & $1$ \\ 
$\mathfrak{a}_{5}$ & $1$ & $0$ & $0$ & $0$ & $0$ & $0$ & $0$ & $0$ & $1$ & $%
1 $ & $1$%
\end{tabular}%
\end{equation*}
So $\mathfrak{f}_{s}\longrightarrow \mathfrak{f}_{s^{\prime }}$, and the
further degenerations are%
\begin{equation*}
\mathfrak{f}_{s}\longrightarrow \mathfrak{a}_{4},\,\mathfrak{f}%
_{s}\longrightarrow \mathfrak{a}_{5},\,\mathfrak{f}_{s}\longrightarrow 
\mathfrak{a}_{0},\,\mathfrak{a}_{4}\longrightarrow \mathfrak{a}_{0},\,%
\mathfrak{a}_{5}\longrightarrow \mathfrak{a}_{0}
\end{equation*}%
as in \cite{IP2}. The ones of the form $\mathfrak{f}_{s}\longrightarrow 
\mathfrak{a}$ combine $\mathfrak{f}_{s}\longrightarrow \mathfrak{f}_{1}$
followed by a contraction from $\mathfrak{f}_{1}$ to $\mathfrak{a}$. A
degeneration $\mathfrak{a}_{5}\longrightarrow \mathfrak{a}_{4}$ is ruled out
by \cite[Lemma 3.17]{IP1}, based on the dimensions. If $\mathbb{F}=\mathbb{C}
$, the derivation dimensions rule out both $\mathfrak{a}_{4}\longrightarrow 
\mathfrak{a}_{5}$ and $\mathfrak{a}_{5}\longrightarrow \mathfrak{a}_{4}$;
see \cite[Section II]{BB} for the algebraic geometric background. For
general infinite $\mathbb{F}$ with $\mathrm{char}\mathbb{F}\neq 2$, one can
find polynomials (using Maple$^{\text{TM}}$) vanishing on one orbit but not
the other:%
\begin{equation*}
\begin{tabular}{cl}
$\lambda _{111}^{2}-\lambda _{212}^{2}+2\lambda _{112}\lambda
_{211}+2\lambda _{112}\lambda _{222}$ & $0$ on $\mathbf{\lambda }(\mathfrak{a%
}_{4})\mathrm{G}$ but not on $\mathbf{\lambda }(\mathfrak{a}_{5})\mathrm{G}$
\\ 
$\lambda _{111}\lambda _{212}-\lambda _{112}\lambda _{211}$ & $0$ on $%
\mathbf{\lambda }(\mathfrak{a}_{5})\mathrm{G}$ but not on $\mathbf{\lambda }(%
\mathfrak{a}_{4})\mathrm{G}$%
\end{tabular}%
\end{equation*}%
Thus neither $\mathfrak{a}_{4}\longrightarrow \mathfrak{a}_{5}$ nor $%
\mathfrak{a}_{5}\longrightarrow \mathfrak{a}_{4}$ holds.

\subsection{Three-dimensional fields}

Let $\mathbb{F}$ be an infinite field whose characteristic is neither $2$
nor $3$. Let $\mathfrak{f}$ be a cubic extension field of $\mathbb{F}$,
considered as a commutative associative algebra over $\mathbb{F}$; the
product $\left[ \mathbf{u},\mathbf{v}\right] $ will be written in the
conventional form $\mathbf{uv}$. Algebra $\mathfrak{f}$ has an identity $%
\mathbf{e}$ and a generating element $\mathbf{w}$ for which $\mathbf{e},\,%
\mathbf{w}$, and $\mathbf{w}^{2}$ form a basis of $V$. By a standard
normalization we may assume that $\mathbf{w}^{3}=p\mathbf{w}+q\mathbf{e}$
for certain $p,\,q\in \mathbb{F}$, with $q\neq 0$. We can further normalize
such an algebra when $p\neq 0$ by taking as generator $\mathbf{u}=\frac{1}{p}%
\mathbf{w}^{2}-\frac{1}{3}\mathbf{e}$ . Then $\mathbf{u}^{3}=\mathbf{u}%
^{2}-(4p^{3}-27q^{2})/27p^{3}$, and $\mathbf{e},\,\mathbf{u}$, and $\mathbf{u%
}^{2}$ are independent. But if $p=0$, we take $b\neq 0$ and $\mathbf{u}%
=1/(9bq)\mathbf{w}^{2}+b\mathbf{w}+1/3$, for which $\mathbf{u}^{3}=\mathbf{u}%
^{2}+(27b^{3}q-1)^{2}/(729b^{3}q)$. In order that $\mathbf{e},\,\mathbf{u}$,
and $\mathbf{u}^{2}$ be linearly independent, it turns out that we need $%
729b^{6}q^{2}\neq 1$. Since $\mathbb{F}$ is infinite, there is a $b$ for
which both the last inequality holds and the constant term in $\mathbf{u}%
^{3} $ is nonzero. Thus we may assume that $\mathfrak{f}$ has the defining
relation $\mathbf{w}^{3}=\mathbf{w}^{2}+c$ for some $c\neq 0$. Denote such
an algebra, field or not, by $\mathfrak{f}_{c}$, with the understanding that
the algebra has an identity $\mathbf{w}^{0}$ and that $\left\{ \mathbf{w}%
^{0},\mathbf{w},\mathbf{w}^{2}\right\} $ is a basis of $V$. (Write $a\mathbf{%
w}^{0}$ just as $a$.) Let $\mathfrak{F}$ be the set of the algebras defined
on $V$ that are isomorphic to any of these field-like algebras on $V$,
including $\mathfrak{f}_{0}$, the one with defining relation $\mathbf{w}^{3}=%
\mathbf{w}^{2}$ (still with basis $\left\{ \mathbf{w}^{0},\mathbf{w},\mathbf{%
w}^{2}\right\} $). We are interested in degenerations $\mathfrak{f}%
\longrightarrow \mathfrak{f}^{\prime }$, with $\mathfrak{f},\,\mathfrak{f}%
^{\prime }$ in $\mathfrak{F}$.

The trace form of $\mathfrak{f}_{c}$ relative to the basis $\left\{ \mathbf{w%
}^{0},\mathbf{w},\mathbf{w}^{2}\right\} $ has matrix%
\begin{equation*}
\begin{bmatrix}
3 & 1 & 1 \\ 
1 & 1 & 3c+1 \\ 
1 & 3c+1 & 4c+1%
\end{bmatrix}%
,
\end{equation*}%
with determinant $-c(27c+4)$. Let $\gamma =27c+4$; this combination shows up
repeatedly in what follows.

We search for isomorphic versions of $\mathfrak{f}_{c}$ by trying to find
elements $\mathbf{u}\in \mathfrak{f}_{c}$ for which $\mathbf{u}^{3}=\mathbf{u%
}^{2}+d$ for some $d$, again aided by Maple$^{\text{TM}}$. Let $\mathbf{u}%
=x_{2}\mathbf{w}^{2}+x_{1}\mathbf{w}+x_{0}\mathbf{w}^{0}$. Then look for
triples $\left\{ x_{0},x_{1},x_{2}\right\} $ for which $\mathbf{u}^{3}-%
\mathbf{u}^{2}\in \mathbb{F}$ by equating the $\mathbf{w}^{2}$- and $\mathbf{%
w}$- coefficients of $\mathbf{u}^{3}-\mathbf{u}^{2}$ to $0$. In solving the
equations, an ingredient that is a root of a quadratic equation appears. The
discriminant of the quadratic needs to be a square in $\mathbb{F}$ to
produce solutions in $\mathbb{F}$. This results in a conic with at least one 
$\mathbb{F}$-point. Its other $\mathbb{F}$-points can be found in the
standard way by taking a line with slope $m$ through the given point and
finding the second intersection with the conic. The resulting coefficients
in $\mathbf{u}$ parameterized by $m$ are as follows, with $\Delta
=m^{2}-3\gamma c$:%
\begin{equation*}
x_{0}=\frac{2c(3m-\gamma )}{\Delta },\,\,x_{1}=\frac{m^{2}-(\gamma
-9c)m+3\gamma c}{\Delta },\,\,x_{2}=\frac{4m}{\Delta }.
\end{equation*}%
Moreover, $d=c(m^{3}-\gamma m^{2}+9\gamma cm-\gamma ^{2}c)^{2}/\Delta ^{3}$.
The determinant of $\mathbf{w}^{0},\,\mathbf{u}$, and $\mathbf{u}^{2}$ is a
nonzero rational function of $m$. It follows that $\mathfrak{f}_{d}\in 
\mathfrak{F}$ for an infinite number of values of $d$ for which $\mathfrak{f}%
_{d}\cong \mathfrak{f}_{c}$. As one computes, $27d^{2}+4d=s^{2}(27c^{2}+4c)$
for a rational function $s$ of $m$. This is in line with the remarks on
trace forms above, since those of $\mathfrak{f}_{c}$ and $\mathfrak{f}_{d}$
should be equivalent.

Let $\mathbf{\lambda }_{c}=\mathbf{\lambda }(\mathfrak{f}_{c})$ be the
structure vector of $\mathfrak{f}_{c}$, $c\neq 0$, relative to a specified
basis of $V$. For $\overline{\mathbf{\lambda }_{c}\mathrm{G}}$ we use the
earlier polynomial argument: if $P$ is a polynomial in the coordinates $%
\lambda _{ijk}$ of the structure space which vanishes on $\mathbf{\lambda }%
_{c}\mathrm{G}$, put $p(x)=P(\mathbf{\lambda }_{x})$, $\mathbf{\lambda }_{x}=%
\mathbf{\lambda }(\mathfrak{f}_{x})$. Then $p(x)$ is a polynomial, and by
the preceding discussion, $p(d)=0$ for infinitely many $d$. Thus $P(\mathbf{%
\lambda }_{x})=0$ for all $x$, and $\mathfrak{f}_{c}\longrightarrow 
\mathfrak{f}_{x}$ for any $x$. Thus

\begin{proposition}
\label{PropFieldClosure}For $c\neq 0$, $\mathbf{\lambda }(\mathfrak{F}%
)\subseteq \overline{\mathbf{\lambda }(\mathfrak{f}_{c})\mathrm{G}}$.
\end{proposition}

The algebras in $\mathfrak{F}$ include any commutative associative
semisimple algebra over $\mathbb{F}$. For if not a field, such an algebra is
isomorphic either to $\mathbb{F}\oplus \mathbb{F\oplus F}$ or to $\mathbb{F}%
\oplus \mathbb{F}(\sqrt{n})$ for some nonsquare $n\in \mathbb{F}$. Let $%
\mathfrak{c}$ be the commutative algebra with basis $\mathbf{v}_{1},\,%
\mathbf{v}_{2},\,\mathbf{v}_{3}$ and relations%
\begin{equation*}
\mathbf{v}_{1}^{2}=\mathbf{v}_{1},\,\mathbf{v}_{2}^{2}=\mathbf{v}_{2},\,%
\mathbf{v}_{3}^{2}=s\mathbf{v}_{2},\,\mathbf{v}_{1}\mathbf{v}_{2}=0,\mathbf{v%
}_{1}\mathbf{v}_{3}=0,\,\mathbf{v}_{2}\mathbf{v}_{3}=\mathbf{v}_{3},
\end{equation*}%
$s\neq 0$, so that $\mathfrak{c}\cong \mathbb{F}\oplus \mathbb{F}(\sqrt{s})$%
, the second summand being $\mathbb{F}\mathbf{v}_{2}+\mathbb{F}\mathbf{v}%
_{3} $. The identity $\mathbf{e}$ of $\mathfrak{c}$ is $\mathbf{v}_{1}+%
\mathbf{v}_{2}$. If $s$ is a nonzero square, then $\mathbb{F}(\sqrt{s})\cong 
\mathbb{F\oplus F}$, via the orthogonal idempotents $\frac{1}{2}(\mathbf{v}%
_{2}\pm \frac{1}{\sqrt{s}}\mathbf{v}_{3})$. To show that $\mathfrak{c}\cong 
\mathfrak{f}_{d}$ for some $d$, we follow the computation strategy we have
been using. The upshot is that for a line slope parameter $m$, we can take $%
\mathbf{w}=x_{1}\mathbf{v}_{1}+x_{2}\mathbf{v}_{2}+x_{3}\mathbf{v}_{3}$, with%
\begin{equation*}
x_{1}=\frac{9m^{2}s-1}{9m^{2}s+3},\,x_{2}=\frac{2}{9m^{2}s+3},\,x_{3}=\frac{%
2m}{3m^{2}s+1},
\end{equation*}%
$m$ chosen to make $\mathbf{e},\,\mathbf{w}$, and $\mathbf{w}^{2}$
independent. The total restriction on $m$ is that $m^{2}s\notin \left\{
-1/3,0,1/9,1\right\} $. Then $\mathbf{w}^{3}=\mathbf{w}^{2}+d\mathbf{e}$,
with%
\begin{equation*}
d=-\frac{4(9m^{2}s-1)^{2}}{27(3m^{2}s+1)^{3}}.
\end{equation*}

\end{document}